\documentclass[runningheads,a4paper]{llncs}

\usepackage{amsmath,amssymb}

\usepackage{graphicx}
\usepackage{color}
\usepackage{bm}
\usepackage{subcaption}

\usepackage{hyperref}
\usepackage{siunitx}

\newcommand{\R}{\mathbb{R}}

\usepackage{url}
\begin{document}

\mainmatter  % start of an individual contribution

\title{Efficient Osmosis Filtering of Thermal-Quasi Reflectography Images for Cultural Heritage}
\titlerunning{Efficient Osmosis Filtering of TQR Images for Cultural Heritage}

\author{Simone Parisotto\inst{1} \and Luca Calatroni\inst{2} \and Claudia Daffara\inst{3}}
\authorrunning{S. Parisotto, L. Calatroni \and C. Daffara}

%\institute{Springer-Verlag, Computer Science Editorial,\\
%Tiergartenstr. 17, 69121 Heidelberg, Germany\\
%\url{http://www.springer.com/lncs}}

\institute{
CCA, Wilberforce Road, CB3 0WA, University of Cambridge, UK\\
\email{sp751@cam.ac.uk}
\and
CMAP, Ecole Polytechnique 91128 Palaiseau Cedex, France\\
\email{luca.calatroni@polytechnique.edu},
\and
University of Verona, Strada Le Grazie 15, 37134 Verona, Italy\\
\email{claudia.daffara@univr.it}
}


\toctitle{Lecture Notes in Computer Science}
\tocauthor{Authors' Instructions}
\maketitle

\begin{abstract}
In Cultural Heritage, non-invasive infrared imaging techniques are used to analyse portions of deep structures behind wall paintings. When mosaicked, these images usually suffer from light inhomogeneities due to the experimental setup, which may prevent restorers from distinguishing the physical properties of the object under restoration. A light-balanced image is therefore essential for inter-frame comparisons, while preserving intra-frames details. In this paper, we apply the image osmosis model proposed in \cite{weickert} to solve the light balance problem in Thermal-Quasi Reflectography (TQR) imaging. Due to the large amount of image data, the computation of the numerical solution of the model may be prohibitively costly. To overcome this issue, we make use of efficient operator splitting techniques. We test the proposed numerical schemes on the TQR measurement dataset of the mural painting ``Monocromo''  by Leonardo Da Vinci at Castello Sforzesco (Milan, Italy). The light corrected result is registered to a visible orthophoto, which makes it re-usable for further restorations.

\end{abstract}

\section{Introduction}
Non-destructive optical infrared imaging techniques are widely used for the scientific study of wall paintings since different spectral bands unveil hidden information about surface and sub-surface layers. For instance, near-infrared reflectography (from \SIrange{0.8}{2.5}{\micro\meter}) is suitable for analysing paintings on canvas and panels, revealing underlying details, preparatory drawings and repainting, while far-infrared thermography (from \SIrange{8}{15}{\micro\meter}) emphasises the underlying deep structures rather than superficial defects. However, traditional reflectography techniques are less effective on wall paintings due to the specific optical properties of the pictorial layers and the plaster support. Recently, thermal mid-infrared band (from \SIrange{3}{5}{\micro\meter}) in reflectance mode has been a powerful tool for analysing pictorial layers of frescoes, undetectable with traditional infrared methods: such imaging technique is called Thermal-Quasi Reflectography (TQR) \cite{Daffara2012}. For very large wall paintings, many TQR imaging shoots need to be taken, each one highly depending on a specific illumination setup.
Due to the fragility of the artworks and to the nature of the materials,  the ideal tuning of the thermal stimulus is often not possible. As a consequence, light inhomogeneities in the TQR mosaic may occur. This unbalance makes the detection of wall materials and the discrimination of their different response very challenging.  Moreover, a light-balanced version of the TQR mosaic is desirable to compare the present status of the wall painting with past and future restorations. 

In order to correct the inter-frame light coefficient, we consider in this paper a model described by a parabolic Partial Differential Equation (PDE), a classical tool in mathematical image processing used to solve several imaging problems such as deblurring, denoising and inpainting (see, e.g.\ \cite{aubert} for a survey).  In particular, following \cite{weickert,vogel,Calatroni2017}, we consider a drift-diffusion PDE  which is the non-symmetric counterpart of classical diffusion processes and which allows for non-trivial steady states. due to its similarities to the analogous physical phenomenon, such model is called \emph{image osmosis}. 

In this work, we discuss the use of a linear image osmosis model as an unsupervised mathematical tool to correct the light differences in TQR frames, preserving intra-frame details at the same time.  Similarly to our previous work \cite{Calatroni2017}, our numerical implementation relies on the use of mathematical operator splitting techniques considered, e.g., in \cite{Barash2001,hundsbook,CalatroniADI} and here designed to be used for possibly very large TQR images. Numerical methods are validated to be efficient and accurate on a TQR mosaic (greater than 28Mpx) of the ``Monocromo" by Leonardo Da Vinci (Castello Sforzesco, Milan, Italy) \cite{Daffara2015}. 

\paragraph{Structure of the Paper.} 
In Section \ref{sec: tqr} we recall the principles of TQR imaging. In Section \ref{sec: linear_osmosis} the linear osmosis model proposed in \cite{vogel,weickert} is described. In Section \ref{sec: adi osmosis} we discuss classical operator splitting techniques and their use as time-discretisation schemes for the osmosis problem. Similarly as in \cite{Calatroni2017}, we prove that such schemes preserve fundamental properties of the continuous osmosis model and can be used as efficient numerical solvers for very large images. In Section \ref{sec: numerical} we solve the light balance problem for a specific TQR case study.

\section{Thermal-Quasi Reflectography}\label{sec: tqr}
Thermal-Quasi Reflectography (TQR) is a modern mid-infrared imaging technique for the analysis of wall painting \cite{Daffara2012} which successfully detects details that standard reflectography methods are not able to measure. The core idea of TQR imaging is the acquisition of the radiation reflected in the mid-infrared thermal range (MIR), from \SIrange{3}{5}{\micro\meter}, which is closely correlated with the material properties on the surface layer.  In standard infrared non-destructive analysis, a thermal camera is used in ``emissive'' modality, recording the electromagnetic radiation emitted by the object to extrapolate the temperature distribution on the surface (thermogram). In this context, the object may be possibly exposed to a thermal stimulus using active or passive thermography techniques. In the case of wall paintings, such approach can only study the deep wall structures and the ``defects'' in the paint layers.
In TQR imaging, the concept of thermography is turned upside down: it is a \emph{quasi-reflectography} technique, where the radiation emitted by the sample is ``rejected'' so as to measure the thermal radiation reflected. 
The key observation at the core of TQR modality is  that an object with constant emissivity and a surface temperature of \SI{293}{\kelvin} emits only $1.1\%$ of its thermal energy in the MIR. This follows by the Plank spectral exitance $M_\lambda$ which quantifies the radiant flux emitted per unit area and wavelength $[\SI{}{\watt\meter^{-3}}]$ via the relation
\begin{equation}   \label{planck}
%M_\lambda={\frac {C_1}{\lambda^5 \bm( {e^{C_2 / \lambda T}-1} \bm)}}
M_\lambda= C_1 \lambda^{-5} ( e^{C_2 / \lambda T}-1)^{-1},
\end{equation}
where $C_1,\,C_2$ are the radiation constants. 
By sending a thermal stimulus in the MIR and limiting the surface heating, a signal dominated by the reflected thermal radiation can then be acquired. The MIR reflectance $r$, defined as the ratio of the incident and reflected radiant flux in the MIR, can then obtained point-wise from the acquired TQR radiometric image $u$ using an in-scene calibration target~by the relation
\begin{equation}   \label{Reflectance}
%u_{ij} = r_{ij} {\frac {u_{\mathrm{ref}}}{r_{\mathrm {ref}}}}
u_{ij} = r_{ij} u_{\mathrm{ref}} r_{\mathrm{ref}}^{-1},
\end{equation}
where $r_\mathrm {ref}$ is the certified reflectance of the reference target and $u_{\mathrm{ref}}$ its averaged radiometric response. 
The calibration target is typically a gold material with Lambertian surface, highly reflective in the MIR. 
An important requirement of the experimental setup should be a uniform radiation source. Furthermore, a variation of the surface temperature increases the emitted radiance, therefore the peak of the Planck spectral distribution \eqref{planck} is shifted towards the MIR. 
Since the MIR reflectance is strictly correlated to the absorption bands of the materials and to the morphology of the surface (micro-roughness), the TQR technique allows the discrimination of different kind of surface materials, e.g.\ highly-reflective pigments, absorbing organic binders or decay products such as salts. 
After TQR acquisition, a MIR reflectogram is produced for mapping different characteristics connected to original materials such as the restoration materials and degradation of wall paintings, even on very extensive areas.
In \cite{Daffara2015}, a non-invasive mid-infrared multi-modal imaging tool has been introduced to effectively detect voids and sub-superficial damages.

As a result of the calibration set up, a non-uniform illumination often affects the mosaic of the TQR frames. For restoration purposes, a balanced version of such mosaic is desirable to have a comprehensive idea of the wall composition and to compare the present status of the wall painting with past restorations. In the following, we consider a mathematical model considered in \cite{weickert} to uniform the
inter-frame light coefficient.

\section{Image Osmosis Filtering} \label{sec: linear_osmosis}
For a regular image domain $\Omega \subset\R^2$ with boundary $\partial \Omega$ and given a vector field $\bm{d}:\Omega\to\R^2$, the linear image osmosis model considered in \cite{vogel,weickert} is a drift-diffusion PDE which computes for every $t\in (0,T],~ T>0$ a family $\left\{u(x,t)\right\}_{t>0}$ of regularised images of a positive initial gray-scale image $f:\Omega\to\R^+$ by solving:
\begin{equation} \label{eq:osmosis}
\begin{cases}
\partial_t u = \Delta u - \text{div}(\bm{d}u) & \text{ on } \Omega \times (0,T] \\
u(x,0)=f(x) & \text{ on } \Omega \\
\langle \nabla u - \bm{d}u, \bm{n} \rangle = 0 & \text{ on } \partial \Omega \times (0, T],
\end{cases}
\end{equation}
where $\langle \cdot, \cdot \rangle$ is the Euclidean scalar product and $\bm{n}$ the outer normal vector on $\partial\Omega$. As proved in \cite[Proposition 1]{weickert}, any solution of \eqref{eq:osmosis} preserves the average gray value (AVG) of $f$ and it is non-negative at any time $t>0$. Furthermore, by setting $\bm{d}:=\ln(\bm{\nabla} v)$ for a given reference image $v>0$, the steady state equation
\begin{equation}   \label{eq:stationary}
\Delta u - \text{div}(\bm{d}u)= 0
\end{equation}
is solved by $w(x)=\frac{\mu_f}{\mu_v}~v(x)$, where $\mu_f$ and $\mu_v$ are the average of $f$ and $v$ over $\Omega$, respectively. In other words, the steady state $w$ of \eqref{eq:osmosis} is a rescaled version of $v$. Such property makes the osmosis model appealing for light-balance applications such as shadow-removal where one wants to uniform the image brightness without losing the initial given information encoded in $f$. In order to avoid the dependence on $v$ (which is not available in practice), the vector field can be set to zero on the shadow boundary and $\bm{d}=\nabla(\ln f)$ everywhere else, see \cite{weickert}. The osmosis evolution compensates for such discontinuity and converges to a steady state which is a rescaled version of $f$ and where the shadow has been removed.

% is a minimiser of the energy functional:
%\begin{equation}   \label{eq:osmosis_energy}
%E(u) = \int_\Omega v \Big| \bm{\nabla} \Bigl(\frac{u}{v}\Bigr) \Big|^2~dx,
%\end{equation}

In \cite{vogel}, a fully discrete theory for the model \eqref{eq:osmosis} is studied. For a given  $\bm{f}\in\R^N_+$, a spatial finite-difference discretisation of the differential operators in \eqref{eq:osmosis} reads
\begin{equation}   \label{eq:semi_discrete}
\bm{u}(0) = \bm{f}, \qquad \bm{u}'(t) =\bm{A} \bm{u}(t),\qquad t>0.
\end{equation}
Equivalently, denoting by $u_{i,j}$ an approximation of $u$ in the grid of size $h$ at point $((i-\frac{1}{2})h, (j-\frac{1}{2})h)$, we can write \eqref{eq:semi_discrete} element-wise as
\begin{align}  \label{eq:discr_operators} 
u'_{i,j} = \bm{A}\bm{u}= &  u_{i,j} \bigg( -\frac{4}{h^2}+\frac{2d_{1,i-\frac{1}{2},j}}{h}-\frac{2d_{1,i+\frac{1}{2}}}{h} +\frac{2d_{2,i,j-\frac{1}{2}}}{h}-\frac{2d_{2,i,j+\frac{1}{2}}}{h} \bigg) \notag \\
&+  u_{i+1,j} \bigg( \frac{1}{h^2}-\frac{d_{1,i+\frac{1}{2},j}}{2h} \bigg)+ u_{i-1,j} \bigg( \frac{1}{h^2}+\frac{d_{1,i-\frac{1}{2},j}}{2h} \bigg)   \\
& + u_{i,j+1} \bigg( \dfrac{1}{h^2}-\dfrac{d_{2,i,j+\frac{1}{2}}}{2h} \bigg)+
u_{i,j-1} \bigg( \dfrac{1}{h^2}+\dfrac{d_{2,i,j-\frac{1}{2}}}{2h} \bigg). \notag
\end{align}
For the time discretisation of \eqref{eq:semi_discrete}, in \cite{vogel,weickert} the authors consider standard forward and backward Euler schemes and prove the validity of analogous conservation and convergence properties of the continuous model \eqref{eq:osmosis} in the discrete setting, \cite[Proposition 1]{vogel}. In particular, in the case of a fully implicit time-stepping unconditional stability is observed.  Computationally, iterative solvers (e.g., BiCGStab) are used to solve the resulting penta-diagonal and non-symmetric linear systems.

\section{Operator splitting schemes}\label{sec: adi osmosis}
The numerical solution the fully discretised osmosis model may become extremely costly for large images. In this section we recall the general framework of operator splitting methods and recall \emph{Alternating Direction Implicit} (ADI), additive and multiplicative splitting schemes used, e.g., for nonlinear diffusion imaging models in \cite{Barash2001,CalatroniADI}.

\subsection{Dimensional splitting schemes} \label{subsec:adi_split}
Given a domain $\Omega\in\R^s$, a \emph{dimensional splitting} method decomposes the space-discretised operator $\bm{A}$ of the initial boundary value problem \eqref{eq:semi_discrete} into the sum:
\begin{equation} \label{decomp}
\bm{A}=\bm{A_0}+\bm{A_1}+\ldots+\bm{A_s}
\end{equation}
where the terms $\bm{A_j}$, $j=1,\ldots, s,$ encode the linear action of $\bm{A}$ along the space direction $j=1,\ldots, s,$ respectively, while $\bm{A_0}$ may contain mixed-derivative contributes and non-stiff nonlinear terms. ADI time-stepping schemes treat the unidirectional components $\bm{A_j},~j\geq 1$ implicitly and the $\bm{A_0}$ component, if present, explicitly in time.  ADI schemes belong to the family of locally-one dimensional (LOD) splitting methods where \eqref{decomp} is used to reduce the $s$-dimensional original problem to $s$ one-dimensional problems. 
%Similarly, when $\bm{A_0}=0$ additive (AOS) and multiplicative (MOS) operator splitting techniques can be used to solve the problem in a fully implicit way, thus ensuring unconditional stability.
Having in mind a standard finite difference space discretisation, LOD methods consider tridiagonal matrices whose inversion can be rendered efficient via classical matrix factorisation techniques.

In our case $s=2$ and no mixed derivatives appear. The tridiagonal operators $\bm{A_1}$ and $\bm{A_2}$ can then be defined from \eqref{eq:discr_operators}  as:
\begin{align}  \label{eq:discr_operators_splitted}
\bm{A_1}\bm{u}&  = u_{i,j} \bigg( -\frac{2}{h^2}+\frac{d_{1,i-\frac{1}{2},j}}{h}-\frac{d_{1,i+\frac{1}{2}}}{h}+\frac{d_{2,i,j-\frac{1}{2}}}{h}-\frac{d_{2,i,j+\frac{1}{2}}}{h} \bigg) \notag \\
& +  u_{i+1,j} \bigg( \frac{1}{h^2}-\frac{d_{1,i+\frac{1}{2},j}}{2h} \bigg)+ u_{i-1,j} \bigg( \frac{1}{h^2}+\frac{d_{1,i-\frac{1}{2},j}}{2h} \bigg);\\
\bm{A_2}\bm{u}&  = u_{i,j} \bigg( -\frac{2}{h^2}+\frac{d_{1,i-\frac{1}{2},j}}{h}-\frac{d_{1,i+\frac{1}{2}}}{h}+\frac{d_{2,i,j-\frac{1}{2}}}{h}-\frac{d_{2,i,j+\frac{1}{2}}}{h} \bigg) \notag \\
 & + u_{i,j+1} \bigg( \dfrac{1}{h^2}-\dfrac{d_{2,i,j+\frac{1}{2}}}{2h} \bigg)+
u_{i,j-1} \bigg( \dfrac{1}{h^2}+\dfrac{d_{2,i,j-\frac{1}{2}}}{2h} \bigg). \notag
\end{align}
%We present in the following two splitting schemes solving \eqref{eq:discr_operators} by means of the splitting \eqref{eq:discr_operators_splitted} in an accurate and efficient way.
\paragraph{The Peaceman-Rachford ADI scheme.} For every $k\geq 0$ and time-step $\tau>0$, the scheme computes an approximation $\bm{u}^{k+1}$ via the updating rule:
\begin{equation}  \label{eq: peacemanrach}
\left\{\begin{aligned}
\bm{u}^{k+1/2}&=\bm{u}^{k}+\frac{\tau}{2}\bm{A_1u}^{k}+\frac{\tau}{2}\bm{A_2u}^{k+1/2},  \\
 \bm{u}^{k+1}&=\bm{u}^{k+1/2}+\frac{\tau}{2}\bm{A_1u}^{k+1}+\frac{\tau}{2}\bm{A_2u}^{k+1/2}.
\end{aligned}\right.  \tag{P-R}
\end{equation}
Note that at each time-step forward and backward Euler are applied alternatively, resulting in a second-order accurate in time semi-implicit scheme \cite{hundsbook}.  
The Peaceman-Rachford iteration \eqref{eq: peacemanrach} satisfies the osmosis conservation properties and converges to a unique steady state only for sufficiently small time steps $\tau<2\left(\max\left\{ \max |a^1_{i,i}| , \max |a^2_{i,i}| \right\}\right)^{-1}$, see \cite[Proposition 1]{Calatroni2017}.

\paragraph{Additive and Multiplicative Operator Splitting.} To achieve stability for every $\tau>0$, we consider here fully-implicit Additive (AOS) and Multiplicative (MOS) Operator Splitting schemes proposed in \cite{Weickert98,Barash2001} for nonlinear diffusion models.
For every $k\geq 0$ and time-step $\tau>0$, the AOS and MOS iterations read:
\begin{align}  
 \bm{u}^{k+1} &= \frac{1}{2}~\sum_{n=1}^2~\left(\bm{I} -2\tau \bm{A_n}\right)^{-1}\bm{u}^{k}, \tag{AOS} \label{eq:AOS} \\
 \bm{u}^{k+1} &= \prod_{n=1}^2~\left(\bm{I} -\tau \bm{A_n}\right)^{-1}\bm{u}^{k}, \tag{MOS}  \label{eq:MOS}
\end{align}
which are stable for any $\tau>0$  and first-accurate in time. In \cite{Barash2001} a more accurate additive-multiplicative (AMOS) combination of \eqref{eq:AOS} and \eqref{eq:MOS} is considered. There, the approximation $\bm{u}^{k+1}$ is computed via the updating rule:
%\begin{equation}  \label{eq:AMOS}
% \bm{u}^{k+1}=\frac{1}{2}\left( \left(\bm{I} -\tau \bm{A_2}\right)^{-1}\left(\bm{I} -\tau \bm{A_1}\right)^{-1} + \left(\bm{I} -\tau \bm{A_1}\right)^{-1}\left(\bm{I} -\tau \bm{A_2}\right)^{-1} \right) \bm{u}^{k}\tag{AMOS}
%\end{equation}
\begin{equation}  \label{eq:AMOS}
 \bm{u}^{k+1}=\frac{1}{2} \sum_{n=1}^2\left( \left(\bm{I} -\tau \bm{A_{j_n}}\right)^{-1}\left(\bm{I} -\tau \bm{A_{i_n}}\right)^{-1}  \right) \bm{u}^{k},\tag{AMOS}
\end{equation}
with $\bm{i}=\{1,2\},\, \bm{j}=\{2,1\}$.
The tridiagonal structure of $\left(\bm{I} -\tau \bm{A_\ell}\right)^{-1}$, $\bm{\ell}=1,2$ allows for efficient inversion by means, for instance, of LU factorisation. 

The following Proposition follows \cite[Proposition 1]{Calatroni2017} and guarantees that the \ref{eq:AOS} and \ref{eq:MOS} iterations applied to the osmosis problem preserve the AVG, the positivity and converge to a unique steady state.

\begin{proposition}  \label{prop:aos_mos}
For an initial image  $\bm{f}\in\R^N_+$, the schemes \eqref{eq:AOS} and \eqref{eq:MOS} with $\bm{A_1}$ and $\bm{A_2}$ as in \eqref{eq:discr_operators_splitted} preserve the AVG, the positivity and converges to a unique steady state for any $\tau>0$.
\end{proposition}

\begin{proof}
The proof is similar to \cite[Proposition 1]{Calatroni2017} and relies on \cite[Propositions 1 and 2]{vogel}. We first note that the operators $\bm{P}^-_i:=\left(\bm{I} -\tau \bm{A_i}\right)^{-1}, ~i=1,2$ are non-negative and irreducible with unitary column sum and positive diagonal entries being each $\bm{A_i}$ a  discretised implicit one-dimensional osmosis operator. Similarly, the operator $\bm{P}_{MOS}:= \bm{P}^-_1\bm{P}^-_2$ is trivially non-negative and irreducible, with unitary column sum by \cite[Lemma 1]{Calatroni2017} and with strictly positive elements on the diagonal, since for $i=1\ldots N$ we have
\[
p_{i,i} = \sum_{k=1}^N (\bm{P}^-_1)_{i,k} (\bm{P}^-_2)_{k,i} =  (\bm{P}^-_1)_{i,i} (\bm{P}^-_2)_{i,i} + \sum_{\substack{k=1,\, k\neq i}}^N (\bm{P}^-_1)_{i,k} (\bm{P}^-_2)_{k,i} >0.
\]
For every $k>0$ the AOS time-stepping  $ \bm{u}^{k+1} = \frac{1}{2}\left(\bm{P}_1^- + \bm{P}_2^-\right)\bm{u}^{k}$ and the MOS one $ \bm{u}^{k+1}= \bm{P}_{MOS} \bm{u}^{k} $ satisfy \cite[Proposition 1]{vogel}. Hence, their iterates $\bm{u}^{k+1}$ preserve the AVG and the positivity for any $\tau>0$. Finally, the unique steady state of \eqref{eq:AOS} and \eqref{eq:MOS} is given by the eigenvector associated to the eigenvalue one of the operator $\bm{P}_{AOS}:=\frac{1}{2}(\bm{P}_1^- + \bm{P}_2^-)$ and $\bm{P}_{MOS}$, respectively.\qed
\end{proof}

%The same result holds trivially for the scheme \eqref{eq:AMOS}, since it is just a combination of the two schemes \eqref{eq:AOS} and \eqref{eq:MOS}.

The same result holds trivially for \eqref{eq:AMOS} by combining \eqref{eq:AOS} and \eqref{eq:MOS}.
\begin{corollary}  \label{cor:amos}
For an initial image  $\bm{f}\in\R^N_+$, the \eqref{eq:AMOS} scheme with $\bm{A_1}$ and $\bm{A_2}$ as in \eqref{eq:discr_operators_splitted} preserve the AVG, the positivity and converges to a unique steady state for any $\tau>0$.
\end{corollary}

Proposition \ref{prop:aos_mos} and Corollary \ref{cor:amos} guarantee that  \eqref{eq:AOS}, \eqref{eq:MOS} and \eqref{eq:AMOS} iterations preserve the fundamental properties of the continuous image osmosis model \eqref{eq:osmosis}. Compared to the second-order accurate \eqref{eq: peacemanrach} scheme, they are only first-order accurate in time, but no restriction is required on the time-step $\tau$, which makes their use appealing for fast convergence.
% Computationally, all such schemes benefit from an efficient inversion of the tridiagonal operators via standard LU matrix factorisation after matrix permutation to reduce the band-width to one (see \cite{Calatroni2017} for considerations on the computational costs).
Moreover, the \eqref{eq:AOS} and \eqref{eq:AMOS} scheme may be accelerated by code parallelisation on $n$.

\section{Workflow and Numerical Results}\label{sec: numerical}
We test the efficient osmosis models on the ``Monocromo'' wall painting by L. Da Vinci (Castello Sforzesco, Milan, Italy). We start our discussion recalling the acquisition workflow. We refer to \cite{Daffara2015} for more technical details.

\subsection{Creation of the Dataset}
For image acquisition, we used two custom power controlled infrared lamps symmetrically positioned at the sides of a high-performance thermal camera FLIR~X6540sc as in Figure \ref{fig: TQR lighting setup}. 
\begin{figure}
\centering
\includegraphics[width=0.6\textwidth]{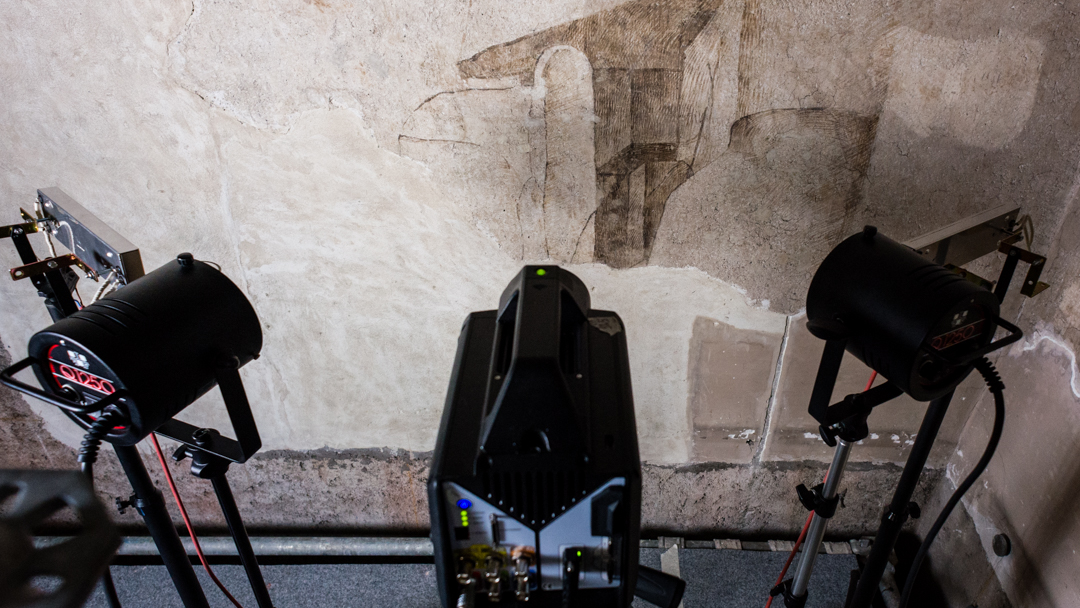}
\caption{Lighting setup: thermal camera paired with infrared lamps.}
\label{fig: TQR lighting setup}
\end{figure}
The emission spectrum of the lamps matches the mid-infrared range and is designed from quartz elements (kanthal alloy filament) and suitable filtering window (sapphire). 
The camera system is equipped with a cooled InSb (indium antimonide) mid-infrared sensitive sensor and an internal filter wheel for selecting thermal bands. Besides high sensitivity and accuracy, this scientific camera allows full access to the raw radiometric data measured by the sensor, (radiance reflected and emitted by the surface). In total, 33 TQR images were sampled to map a section of the wall fresco, see Figure \ref{fig: visible orthophoto}. 

\subsection{Pre-Processing} 

\paragraph{Exporting raw data in Matlab by Python.}
We used the thermal camera as a radiometer: the raw data registered by the sensor  were exported from \texttt{.ptw} and \texttt{.fcf} to Matlab \texttt{.mat} format by means of the Pyradi toolbox \cite{Willers2012}.

%\inputpython{code/ptw2mat.py}{Converting \texttt{.ptw} of TQR to \texttt{.mat}.}
%\inputpython{code/fcf2mat.py}{Converting video frames \texttt{.fcf} of thermal registration to \texttt{.mat}.}

\paragraph{Registration with orthophoto.} In order to facilitate the workflow for future work restorations, we manually registered all the TQR images to an orthophoto (provided by \href{http://www.culturanuova.it/}{Culturanova S.r.l.}), obtaining the TQR mosaic in Figure \ref{fig: TQR mosaic}. To overcome camera lens distortions and parallax errors, we used the \emph{Control Point Selection Tool} in MATLAB (\texttt{cpselect.m}) and a the projection \texttt{fitgeotrans.m} part of the \emph{Imaging Processing MATLAB Toolbox}.
\begin{figure}
\centering
\includegraphics[width=0.3\textwidth]{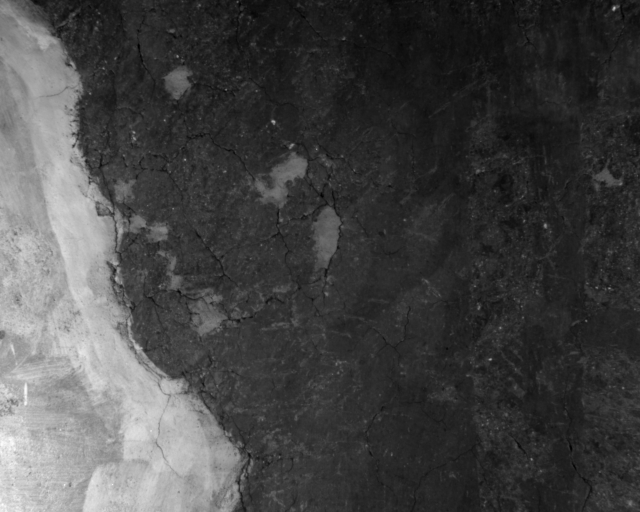}
\hfill
\includegraphics[width=0.3\textwidth]{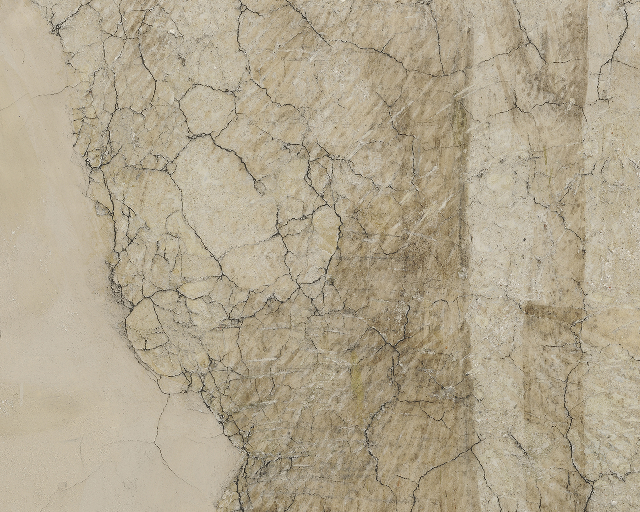}
\hfill
\includegraphics[width=0.3\textwidth]{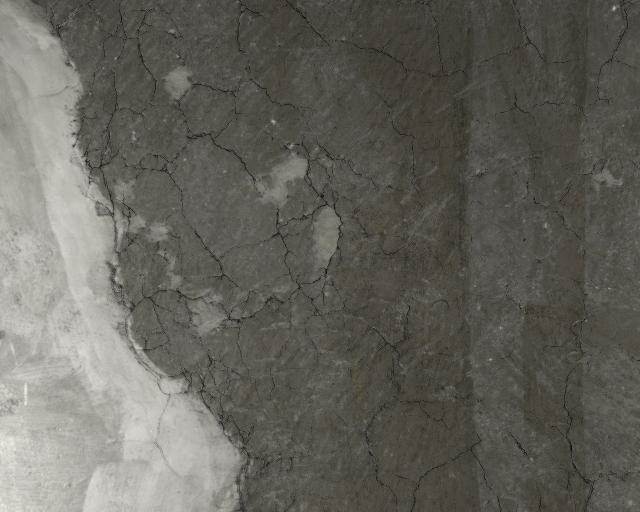}
\caption{Left: TQR; Middle:  Orthophoto aligned with TQR; Right: Alpha fusion.}
  \label{fig:ortho_TQR}
\end{figure}
The large size of the resulting TQR mosaic (4717 $\times$ 7066 pixels) makes the design of unsupervised efficient light correction algorithms challenging. A fast numerical method converging to a light balanced TQR image in a small number of iterations is useful for in-situ study of the composition of the wall and for comparisons with past restorations.

\subsection{Light correction via Efficient Osmosis Filtering}
In order to balance the light differences in Figure \ref{fig:ortho_TQR}, image osmosis filtering can be applied. For an efficient computation, we consider the  \eqref{eq:AOS}, \eqref{eq:MOS} and \eqref{eq:AMOS} splitting schemes. In Figure \ref{fig:mandrill} we study their numerical efficiency and accuracy when used to solve the standard image osmosis problem \eqref{eq:osmosis} for the classical \texttt{mandrill} image. We compare our results with \eqref{eq: peacemanrach} scheme (see \cite{Calatroni2017}) and with a non-split direct implicit solver based on the \texttt{LU} factorization of the penta-diagonal matrix $\bm{A}$. We do not compare here the results obtained via the \texttt{BiCGStab} iterative solver in \cite{weickert} since they highly depend on the tolerance and number of iterations fixed by the user. For efficiency improvements and to avoid memory shortage due to the large size of the problem, a permutation on the tridiagonal matrices $\bm{A_1}$ and $\bm{A_2}$ reducing the band-width to one is used. In Figure \ref{fig: ADI loglog} the  convergence-in-time:  the splitting approaches dramatically reduce the computational times. Note that the second-order accuracy of Peaceman-Rachford scheme \eqref{eq: peacemanrach} is guaranteed only for small step-sizes $\tau$, while first-order accurate \eqref{eq:AOS}, \eqref{eq:MOS} and \eqref{eq:AMOS} schemes allow arbitrarily large time-steps and fast convergence. 

\begin{figure}\centering
\begin{subfigure}{0.5\textwidth}
\centering
\includegraphics[width=1\textwidth]{\detokenize{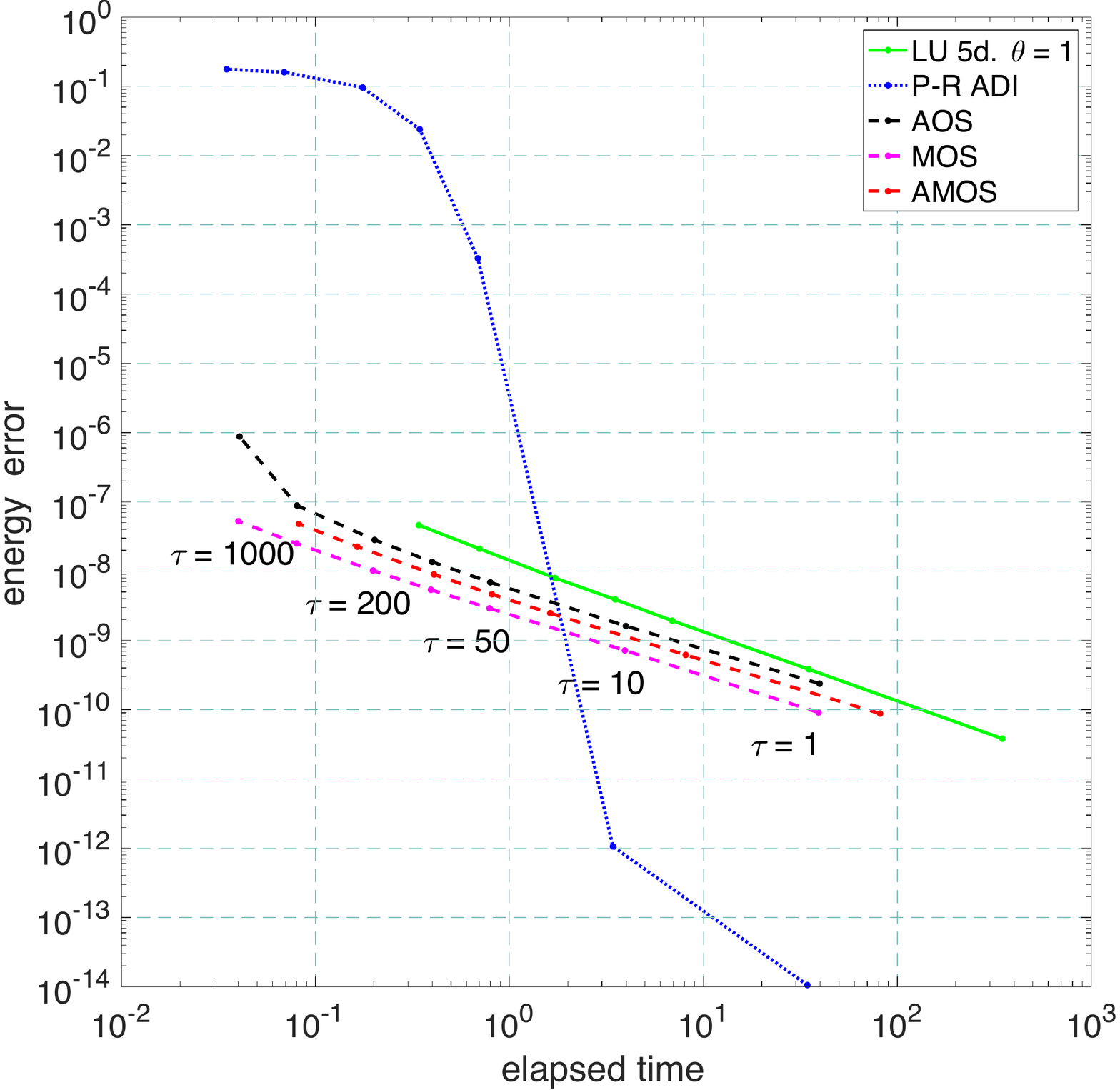}}
\caption{Elapsed Time vs.\ Energy Error.}
\label{fig: ADI loglog}
\end{subfigure}\hfill
%\begin{subfigure}{0.32\textwidth}
%\centering
%\includegraphics[height=3cm]{\detokenize{./images/paper_mask.png}}
%\caption{RMS error.}
%\label{fig:accuracy}
%\end{subfigure}
\begin{subfigure}{0.45\textwidth}
\centering
\includegraphics[width=0.82\textwidth]{\detokenize{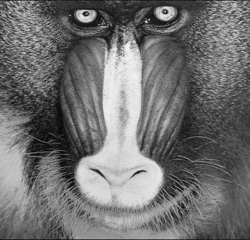}}
\caption{AOS Mandrill, $\tau=1000$.}
\label{fig:mandrill}
\end{subfigure}
\caption{
Numerical efficiency and accuracy of \eqref{eq:AOS}, \eqref{eq:MOS}, \eqref{eq:AMOS} schemes compared with \eqref{eq: peacemanrach} and implicit \texttt{LU} factorisation of penta-diagonal matrix $\bm{A}$ in \eqref{eq:semi_discrete}, see \cite{Calatroni2017}. For Figure \ref{fig: ADI loglog}: $\tau\in \{1,10,50,100,200,500,1000\}, ~T=5000$.}
\end{figure}

In Figure \ref{fig: TQR BEFORE AFTER} we consider the TQR mosaic and report the solution of the light balance problem computed using \eqref{eq: peacemanrach}, \eqref{eq:AOS}, \eqref{eq:MOS} and \eqref{eq:AMOS} schemes with nested \texttt{LU} factorisation of the tridiagonal matrices. The computational times needed to reach an approximation of the steady state at $T=1\mathrm{e}5$ are also reported. The time-step is chosen as $\tau=1000$. Despite the \eqref{eq: peacemanrach} iterations appear to be the most efficient, we observe that the large choice of the time-step $\tau$ badly affects accuracy producing several artefacts in the reconstruction. The \eqref{eq:AOS} and \eqref{eq:MOS} solutions do not show this drawback (see Proposition \ref{prop:aos_mos}) and turn out to be equally efficient for this problem. Similarly, the \eqref{eq:AMOS} scheme does not show any reconstruction artefact, but it suffers from poorer efficiency since it requires the solution of twice the number of linear systems via \texttt{LU} factorisation.

\begin{figure}[t]
\centering
\begin{subfigure}{0.29\textwidth}
\includegraphics[width=1\textwidth]{\detokenize{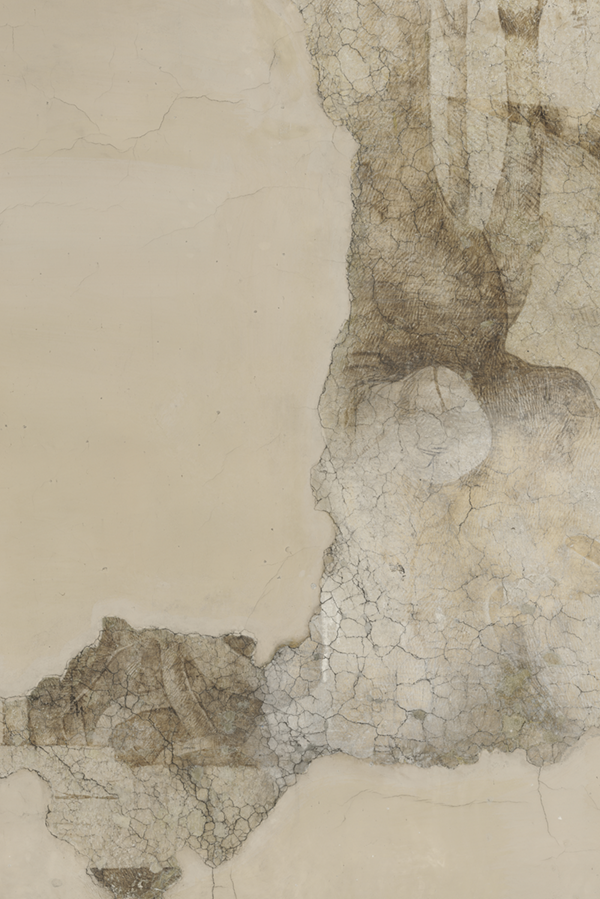}}
\caption{Visible Orthophoto.}
\label{fig: visible orthophoto}
\end{subfigure}
\hfill
\begin{subfigure}{0.29\textwidth}
\includegraphics[width=1\textwidth]{\detokenize{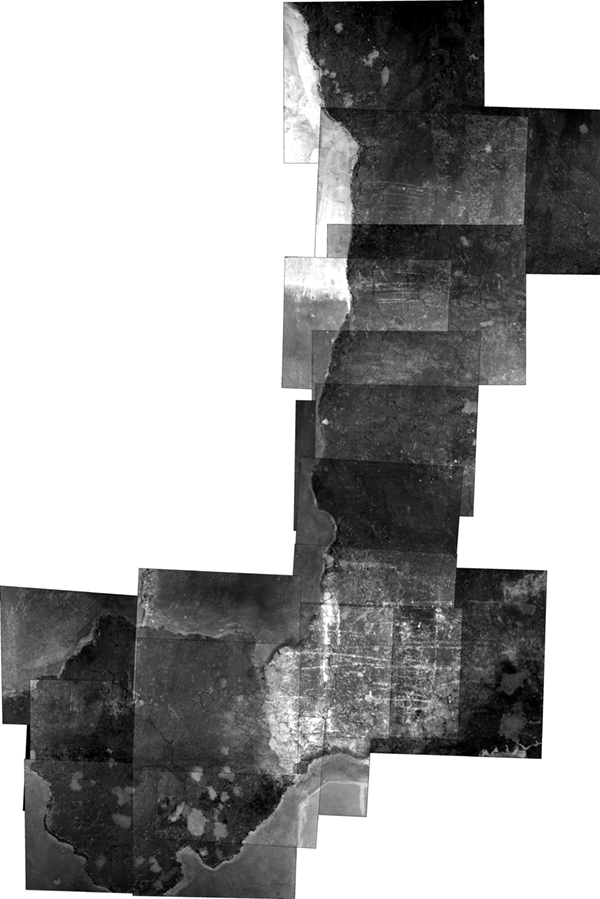}}
\caption{TQR mosaic.}
\label{fig: TQR mosaic}
\end{subfigure}
\hfill
\begin{subfigure}{0.29\textwidth}
\includegraphics[width=1\textwidth]{\detokenize{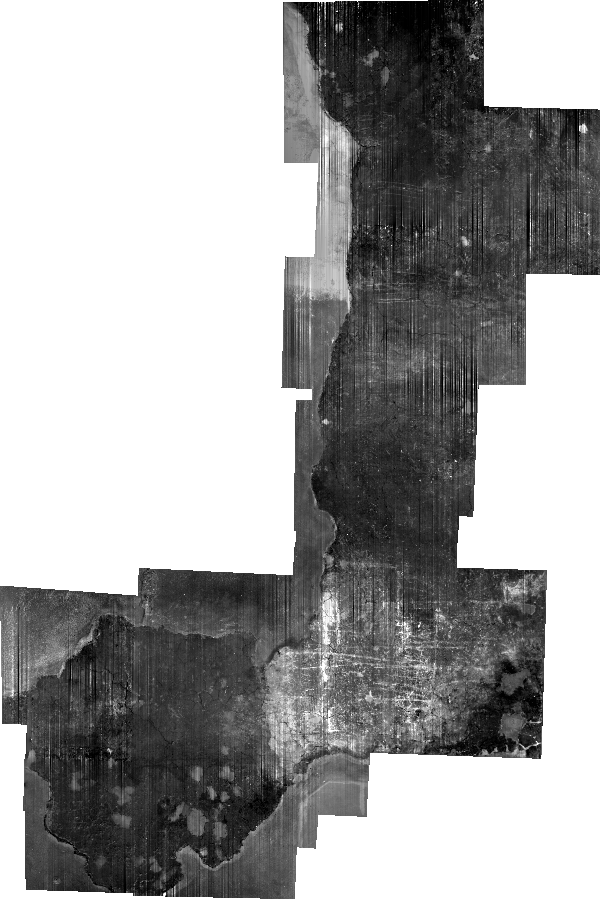}}
\caption{P-R (\SI{526}{\second})}
\label{fig: TQR Peaceman Rachford}
\end{subfigure}
\\
\begin{subfigure}{0.29\textwidth}
\includegraphics[width=1\textwidth]{\detokenize{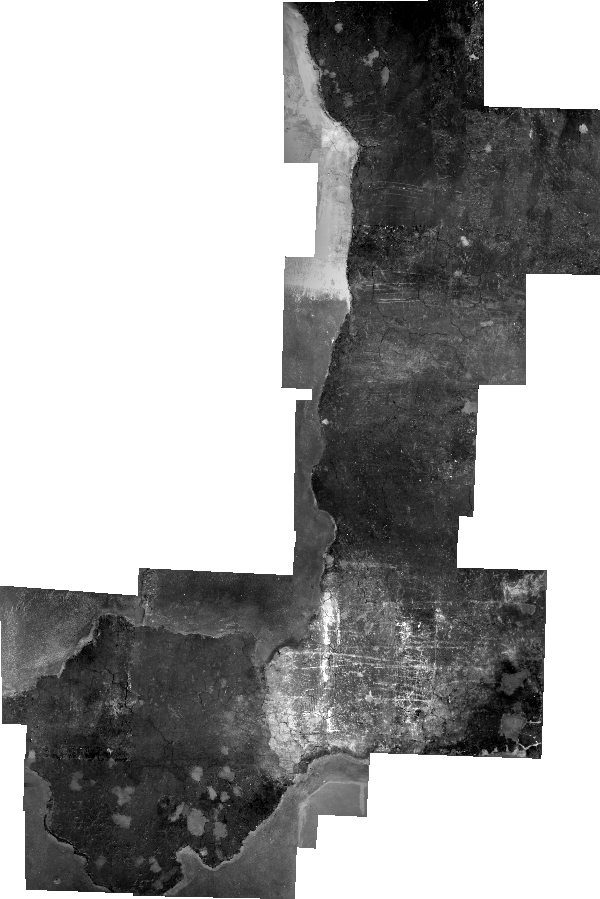}}
\caption{AOS (\SI{629}{\second}).}
\label{fig: TQR AOS}
\end{subfigure}
\hfill
\begin{subfigure}{0.29\textwidth}
\includegraphics[width=1\textwidth]{\detokenize{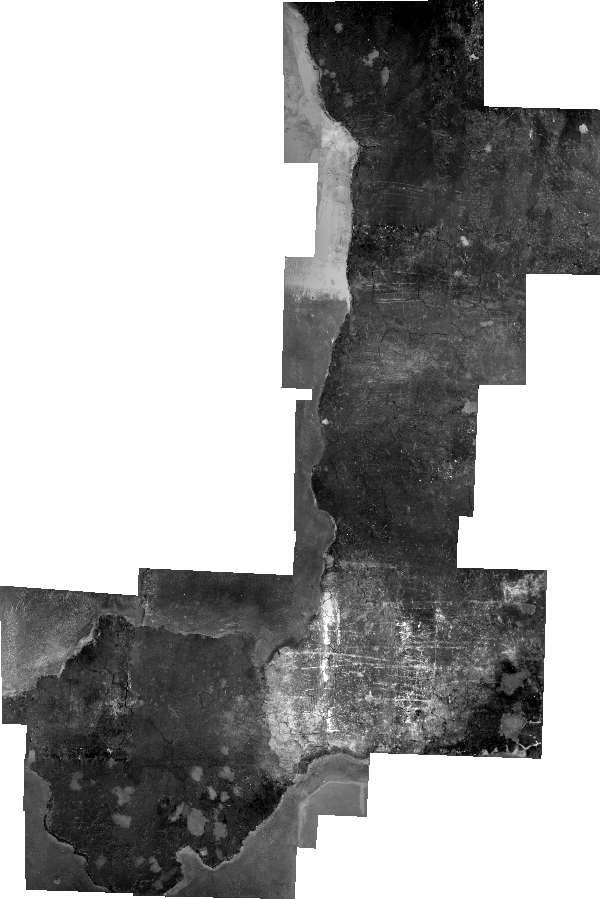}}
\caption{MOS (\SI{629}{\second}).}
\label{fig: TQR MOS}
\end{subfigure}
\hfill
\begin{subfigure}{0.29\textwidth}
\includegraphics[width=1\textwidth]{\detokenize{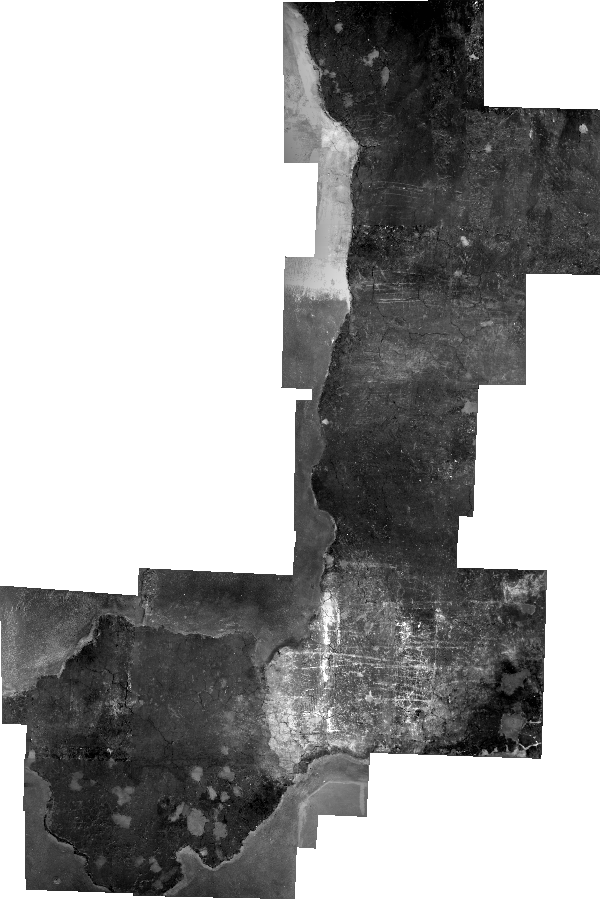}}
\caption{AMOS (\SI{1279}{\second}).}
\label{fig: TQR AMOS}
\end{subfigure}
\caption{TQR light-balance results via osmosis splitting schemes: $T=1\mathrm{e}5$, $\tau=1\mathrm{e}3$.}
\label{fig: TQR BEFORE AFTER}
\end{figure}

\section{Conclusions and Outlook}\label{sec: conclusion}
In this paper we consider an application of the linear image osmosis model proposed in \cite{weickert} as a light balance tool for TQR imaging for cultural heritage. For an efficient computation of their numerical solution, we considered  numerical schemes based on the splitting of the differential operators of the semi-discretised problem along the two spatial dimensions. The resulting fully implicit time-discretised schemes are stable and preserve the standard conservation and convergence properties of the continuous osmosis model for every time-step $\tau>0$. The efficiency and accuracy of the schemes considered make their use suitable for large TQR image data. 

Osmosis filtering may be applied to other diagnostic imaging techniques based on reflectance acquisition mode with either penetrative or non-penetrative wavelengths. Examples of the former techniques are UV reflectance imaging and reflectance transformation imaging (RTI), while an example of the latter technique is near-infrared reflectography for which the diffuse reflectance depends also on the internal scattering across the painting layers.

\subsubsection*{Acknowledgements.} 
%The authors are grateful to Francesco Messa (\emph{FLIR Systems}) for the support in the thermal measurements and to Dr.\ Paola Mariotti (\emph{Opificio delle Pietre Dure}) for the validation of the results. 
SP acknowledges UK EPSRC grant EP/L016516/1 for the University of Cambridge, Cambridge Centre for Analysis DTC.
LC acknowledges the joint ANR/FWF Project \emph{EANOI} FWF n. I1148/ANR-12-IS01-0003.
The diagnostics was supported by Dr.\ Francesca Tasso (\emph{Soprintendenza} of Castello Sforzesco).

\subsubsection*{Data Statement} 
The data leading to this publication are sensitive; restricted access is subjected to the approval of Castello Sforzesco's administration ``Soprintendenza Castello, Musei Archeologici e Musei Storici'', Milan.

\end{document}